\documentclass[11pt]{article}
\usepackage[utf8]{inputenc}
\usepackage[margin=2.5cm]{geometry}
\usepackage{amsmath,amsthm,amsfonts,amssymb}
\usepackage[mathscr]{euscript}
\usepackage{amsmath}
\usepackage{mathrsfs}
\usepackage[version=4]{mhchem}
\usepackage{placeins}
\usepackage{array,float}

\usepackage{graphicx,latexsym}
\usepackage{lscape}
\usepackage{multicol}
\usepackage{subfig}
\usepackage{graphicx}
\usepackage{multirow}
\usepackage{graphicx}
\usepackage{enumitem}
\usepackage{makecell}
\usepackage{cite}

\newtheorem{thm}{Theorem}

\newtheorem{conj}{Conjecture}
\newtheorem{lem}[thm]{Lemma}

\newtheorem{discu}{Discussion:}

\begin{document}
	\title{\textbf{Metric Dimension of Villarceau Grids}}
	\author{\begin{tabular}{cccc}
			\textbf{S. Prabhu$^{\text a, }$\thanks{Corresponding author: drsavariprabhu@gmail.com}, D. Sagaya Rani Jeba$^{\text b}$, Paul Manuel$^{\text c}$, Akbar Davoodi$^{\text d}$} \\
		\end{tabular}\\
		\begin{tabular}{cccc}
{\small $^{\text a}$Department of Mathematics, Rajalakshmi Engineering College, Chennai 602105, India} \\
{\small $^{\text b}$Department of Mathematics, Panimalar Engineering College, Chennai 600123, India } \\
{\small $^{\text c}$Department of Information Science, College of Life Sciences, Kuwait University, Kuwait } \\
$^{\text d}$\small The Czech Academy of Sciences, Institute of Computer Science\\ \small Pod Vodárenskou věží 2, 182 07 Prague, Czech Republic\\
	\end{tabular}}
	\maketitle
	\vspace{-0.5 cm}
	\begin{abstract}
 The metric dimension of a graph measures how uniquely vertices may be identified using a set of landmark vertices. This concept is frequently used in the study of network architecture, location-based problems and communication. Given a  graph $G$, the metric dimension, denoted as $\dim(G)$, is the minimum size of a resolving set, a subset of vertices such that for every pair of vertices in $G$, there exists a vertex in the resolving set whose shortest path distance to the two vertices is different. This subset of vertices helps to uniquely determine the location of other vertices in the graph. A basis is a resolving set with a least cardinality. Finding a basis is a problem with practical applications in network design, where it is important to efficiently locate and identify nodes based on a limited set of reference points. The Cartesian product of $P_m$ and $P_n$ is the grid network in network science. In this paper, we investigate two novel types of grids in network science: the Villarceau grid Type I and Type II. For each of these grid types, we find the precise metric dimension.
		{\small }
	
	\end{abstract}
\section{Introduction}
 Graphs are widely used to model relationships between different entities and play a fundamental role in various disciplines, including computer science, operations research and biology. Let $G$ be a connected graph having vertex set $V(G)$. For $r >0$  and $s\in V(G)$, $N_r(s)$ denotes the set of vertices of $G$ that are at distance  $r$ from $s$. The metric dimension provides a quantitative measure of how effectively a set of vertices can function as landmarks to uniquely determine the location or identity of other vertices in the graph. 
  \begin{equation*}
\mathscr{C}_{\mathscr{R}}(s)=(d(s,r_{1}), d(s,r_{2}), \ldots, d(s,r_{l}))
\end{equation*}
 is the $ l $-vector that represents the code of $s \in V(G)$ with respect to $\mathscr{R}= \{r_{1},r_{2},\ldots,r_{l}\}$. The set $\mathscr{R}$ is a resolving set for $G$ if different vertices of $G$ have distinct codes with regard to $\mathscr{R}$. The lowest cardinality of a resolving set for $G$ is denoted as $\dim(G)$, the resolving number or dimension. This idea has been significantly explored in combinatorics and numerous computer science applications. This graph's dimension is a fundamental parameter that measures the bare minimum of vertices required to uniquely determine the position of every other vertex in the graph, according to their distances to a selected set of reference vertices.\\
 This parameter is used in diverse domains, including robotic navigation \cite{KhRaRo96}, chemistry \cite{ChErJo00, Jo93}, image processing, pattern recognition\cite{MeTo84}, linked joins in graphs \cite{SeTa04}, geometrical routing protocols \cite{LiAb06, MeTo84, Go04} and network intruder detection\cite{BeEbEr06}. Numerous intriguing similarities between the coin weighing and mastermind game have been demonstrated in \cite{SoSh63, CaHeMo07}. 
  Practical scenarios often employ heuristics and approximation algorithms to determine or estimate this parameter in practical scenarios.\\
  This investigation traces back to the early 1970s, with one of the seminal papers by Harary and Melter \cite{HaMe76} in 1976, motivated by challenges in chemical graph theory and communication networks. Initially introduced by Slater in graphs, the concept gained more attention as researchers explored its properties and applications. As soon as researchers recognized the practical applications in numerous domains, there was a surge in interest in algorithmic approaches and heuristics for computing or approximating the metric dimension. Several papers during this period focused on efficient algorithms and approximation techniques.\\
  Despite its NP-hard complexity for arbitrary graphs, significant strides have been achieved in solving the metric dimension problem for particular classes of graphs, leveraging efficient algorithms and approximation techniques. Beyond theoretical advances, practical applications of the metric dimension and its variants continue to emerge, underscoring its relevance across various disciplines.\\
The resolving-power domination number is examined in \cite{PrDeAr22} for probabilistic neural networks, a specific type of RNN (Recurrent Neural Networks). The performance of parallel computing systems on a grand scale is significantly influenced by the quality of the interconnection networks. The various processing components that comprise up a parallel computing system rely extensively on interconnection networks to exchange data and communicate with one another. The metric size and fault tolerance of fractal cubic networks were computed by Arulperumjothi et al.~\cite{ArKlPr23}. An irregular network with a convex triangular metric size is computed in \cite{PrSaAr23}. Prabhu et al.~\cite{PrMaAr20} refute the earlier findings from \cite{HaKhMa20} and \cite{RaHaIm19} by re-examining the metric dimension variant for butterfly networks, silicate networks, and Benes networks.   
Furthermore, a variety of architectures, notably generalized fat trees \cite{PrMaDa24}, trees \cite{KhRaRo96}, torus networks \cite{MaRaRa06}, 2-trees \cite{BeDaJa17}, multi-dimensional grids \cite{KhRaRo96, KeTrYe18}, Benes networks \cite{MaAbRa08}, hypercube \cite{Be13}, honeycomb networks \cite{MaRaRa08, RaThMo11}, Illiac networks \cite{RaRaGo14}, enhanced hypercubes \cite{RaRaMo08},  Petersen graphs \cite{ShShWu18} and circulant graphs \cite{ImBaBo12,  GrMiRa14, Ve17}  have been analysed. This encompasses various types of graphs, such as Toeplitz graphs \cite{LiNaSi19}, infinite graphs \cite{CaHeMo09}, Cayley graphs \cite{FeGoOe06}, Kneser and Johnson graphs\cite{BaCa13},  Cartesian product graphs\cite{CaHeMo07}, chain graphs \cite{FeHeMe15}, fullerene graphs \cite{AkFa19}, Grassmann graphs\cite{BaMe11}, convex polytopes \cite{ImSi16, KrKoCa12}, permutation graphs \cite{HaKaYi14}, bilinear graphs \cite{FeWa12, GuWaLi13}, wheel-related graphs\cite{SiIm14}, incidence graph \cite{Ba18}, regular graphs \cite{GuWa13} and unicyclic graphs \cite{PoZh02}. Numerous details about this topic are available in \cite{TiFrLl23}.   \\ 
A grid is a graph having vertices $\{v_{i,j}:(i,j)\in [\![1,m]\!] \times [\![1,n]\!]\}$ and edges  $\{v_{i,j}v_{i',j'}: |i-i'|+|j-j'|=1\}$. This is $m\times n$ grid, also known as Cartesian product $P_m \square  P_n$ \cite{CrOs15}. There are different grid configurations, including square, hexagonal, triangular, and rectangle ones. Because of their unique characteristics and uses, each kind of grid is appropriate for a variety of jobs and applications. Grids are a useful tool for organising data in spreadsheets, representing spatial data in geographic information systems, and providing a foundation for alignment and layout in graphic design. Mathematicians, computer scientists, and network analysts frequently utilise grid graphs to simulate spatial relationships, connectivity, and algorithms. In computer networks, grids are important because they offer an organised framework for resource management, communication, and service optimisation. Enhancing network performance, scalability, stability, and efficiency in diverse networked contexts is possible through the use of grid-based approaches and algorithms. Tomescu and Melter \cite{MeTo84} examined the metric size for grid graphs where the distances are defined in $L_1$ metrics. They demonstrated that, under the $L_1$ metric, the metric size is two, whereas under the $L_\infty$ metric, it is three. Under the $L_1$ metric, Kuller \cite{KhRaRo96} yields a characterization of the metric size of $d$-dimension. The metric basis is provided in \cite{SaBaTi22} for the square grid graph. The study of graceful labelling for the grid graph has been investigated by Acharya and Gill \cite{AcGi81}. This article examines dimensions for novel grid structures, such as the Villarceau grids.
  	
\section{Villarceau Grids}	
The Villarceau circles \cite{LeHe14,Ma23} are a set of circles formed by the intersection of a plane inclined at a specific angle to the axis of a torus. They are named after the 19th-century French mathematician Yvon Villarceau, who studied and described these circles. When a plane intersects a torus (a surface shaped like a doughnut or a ring) at a non-perpendicular angle to the torus axis, these circles, known as the Villarceau circles, form a distinctive and captivating pattern. They are essential in understanding the geometry and topology of toroidal shapes and find applications across various mathematical and scientific disciplines. The properties of the Villarceau torus are studied in depth by Manuel \cite{Ma23}.\\ 
The Villarceau torus is a geometric configuration that arises from the intersection of a torus with a plane at a specific inclination to the torus's axis. It provides insights into the geometric relationships and topological properties of toroidal shapes, making it significant in geometry and topology.
This mathematical concept has applications in various fields and provides a unique perspective on the interaction between planes and tori. Figure \ref{VT1} illustrates the Villarceau torus, where the red and green circles, also known as the Villarceau circles, have an oblique inclination compared to the white and yellow circles of the torus.  The Villarceau grids are motivated by the Villarceau torus. 
\begin{figure}[H]
	\centering
\includegraphics[scale=0.3]{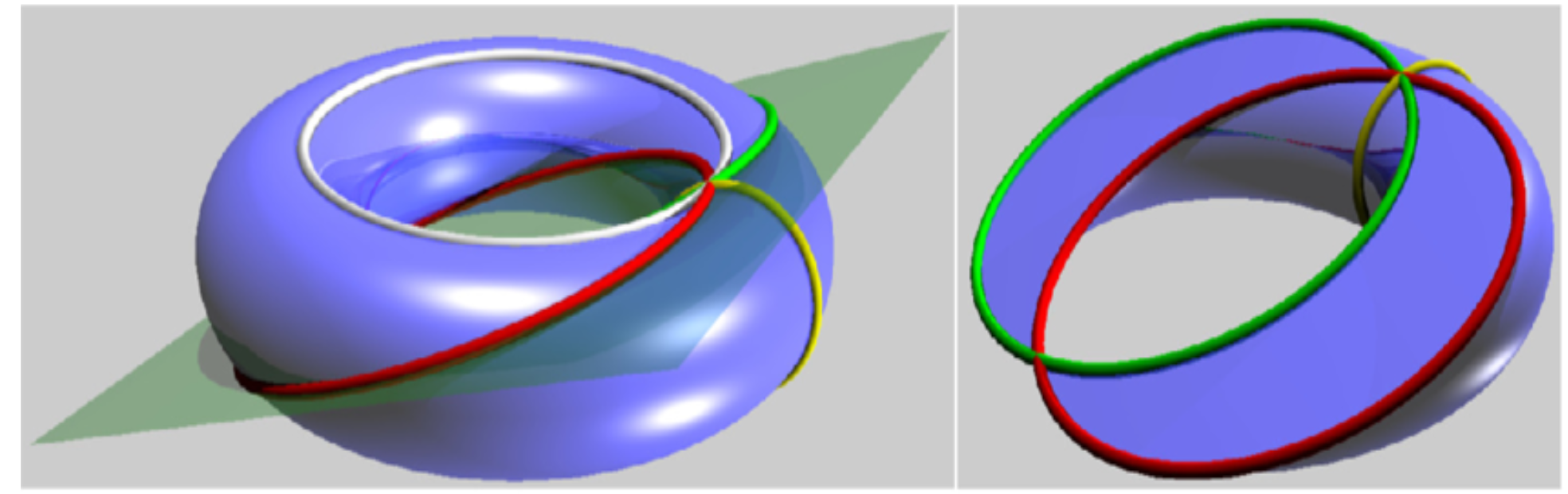}
	\caption{Villarceau torus }
	\label{VT1}
\end{figure}

Our understanding of the Villarceau torus has empowered us to define the Villaceau grid, a novel grid layout distinct from traditional designs. The study of the Villarceau grid enhances our comprehension of complex geometric structures and enriches mathematical exploration. \\
The Villarceau grid Type I is denoted by $VG^1_{m,n}$, where $1 \leq m\leq n$. The vertex set is defined by $\{ (2i,2j+1) : i\in [\![0,n]\!], j \in [\![0,m-1]\!] \} \bigcup \{ (2i+1,2j) : i\in [\![0,n-1]\!], j \in [\![0,m]\!]   \}$ and there is an edge between $(i_1,j_1)$ and $(i_2,j_2)$  if $|i_1-i_2|=1 $ and $ |j_1-j_2|=1$. The notation $[\![ r, s]\!]$ will be utilised for expressing the range of numbers $r<s$, $\{r, r+1, \ldots, s\}$.
\begin{figure}[H] 
	\centering
	\subfloat[]{\includegraphics[scale=.90]{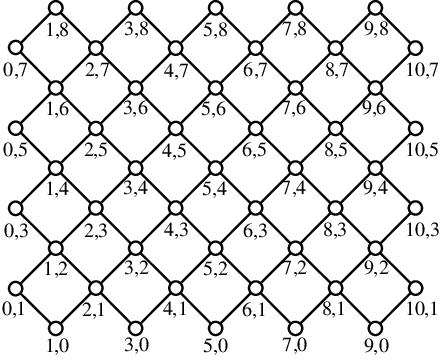}} 
	\quad \quad \quad \quad \quad \quad 
	\subfloat[]{\includegraphics[scale=.70]{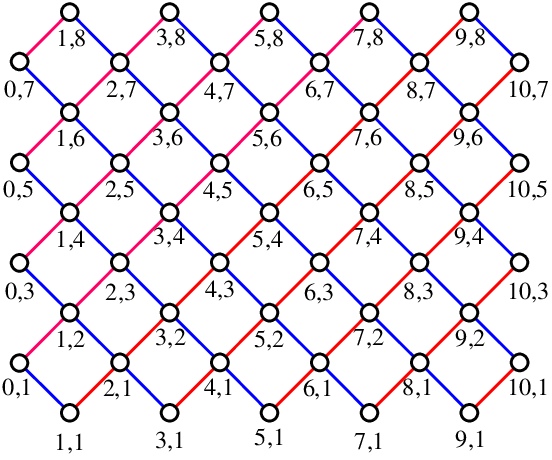}}
	\quad \quad \quad \quad \quad 
	\caption{(a) $VG^1_{4,5}$; (b) Acute lines are depicted in red, while obtuse lines are depicted in blue.} \label{F2}
\end{figure}
The Villarceau grid Type II is  denoted as  $VG^2_{m,n}$, where $1\leq m < n$ (Figure \ref{F3}). The vertex set is defined by $ \{ (2i+1,2j+1) : i\in [\![0,n-2]\!], j \in [\![0,m-1]\!] \} \bigcup \{ (2i,2j) : i\in [\![0,n-1]\!], j \in [\![0,m]\!] \} $ and the vertices $(i_1,j_1)$ and $(i_2,j_2)$ are adjacent if $|i_1-i_2|=1 $ and $ |j_1-j_2|=1$.

\begin{figure}[H]
	\centering
	\includegraphics[scale=0.9]{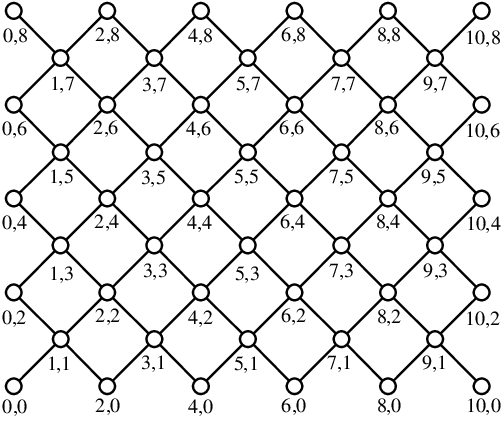}
	\caption{$VG^2_{4,6}$ }
	\label{F3}
\end{figure}
A comparison between the standard grid and the Villarceau grid is shown in the following table.
\begin{table}[H] 
	\centering
	\caption{Grid and Villarceau grid architectures comparison }\label{TI}
	\begin{tabular}{|l|l|l|l|} \hline 
		\multicolumn{1}{|c|}{   } &\multicolumn{1}{c|}{Grid $(P_m \square P_n)$} &\multicolumn{1}{c|}{Villarceau grid Type I $(VG^1_{m,n})$} &\multicolumn{1}{c|}{Villarceau grid Type II $(VG^2_{m,n})$}\\ \hline
		Number of vertices   & $mn$  & $2mn+m+n$ & $2mn-m+n$\\ \hline
  Number of edges   & $2mn-m-n$  & $4mn$ & $4m(n-1)$\\ \hline
  Diameter   & $m+n-2$  & $2n$ & $2n$\\ \hline
  Degree of vertices   & $2, 3, 4$  & $2, 4$ & $1, 2, 4$\\ \hline
  Metric dimension   & $2$  & $  \left\{ \begin{array}{rll}
	 =3 & \mbox{for} \enspace n \leq 2m+1\\
	 > 3 &\mbox{for} \enspace n > 2m+1  \\
	 \end{array}\right.$ & $  \left\{ \begin{array}{rll}
	 =3 & \mbox{for} \enspace n \leq 2m+1\\
	 > 3 &\mbox{for} \enspace n > 2m+1  \\
	 \end{array}\right.$\\ \hline
Average distance   & \makecell{$(m^3n^2-mn^2+$\\$m^2n^3-m^2n)/$ \\$3(m^2n^2-mn)$}  & \makecell{$(10n^3-10n+40m^2n^3+$\\$40mn^3+60m^2n^2+30mn^2+$\\$30m^2n+20m^4n+40m^3n  -$\\$20mn-4m^5+10m^3)/$ \\$15(4m^2n^2+4m^2n+4mn^2+$\\$m^2+n^2-m-n)$} & \makecell{$2(15m^2n+20m^3n+10m^4n-$\\$2m^5-10m^4-5m^3-$\\$20m^3n^2+5n^3-20m^2n^2+$\\$10m^2-15mn^2+7m+$ \\$20n^2m^3+20n^3m^2-10n^2m^2+$\\$20mn^3-10mn-5n)/$  \\$15(4m^2n^2-4m^2n+4mn^2+$\\$-4mn+m^2+m+n^2-n)$}\\ \hline
	\end{tabular}
	\label{Table1}
\end{table}
\section{Main Results}	
\begin{thm}{\rm\cite{KhRaRo96}}\label{md2} 
If a graph $G$'s metric basis is $\{s, t\}$ , then the following criteria are satisfied.
\begin{enumerate}[label={\rm (\roman*)}]
 \item The vertex $s,t$ in $G$ has a degree of $3$ or less.
 \item Between vertices $s$ and $t$ in $G$, there is precisely a shortest path $P$.
\item  The degree of an internal vertex of $P$ in $G$ is less than or equal to $5$.
\end{enumerate}
\end{thm}

 For $n = m = 1$, $VG^1_{1,1} \cong  C_4$ where $C_4$ is a cycle with 4 vertices whose metric dimension is 2, which is discussed in \cite{KhRaRo96}.
 
\begin{lem} \label{lem1}
If $2 \leq n\leq 2m+1$ and $m\neq n$, then $ \dim(VG^1_{m,n}) >2$.
 \end{lem} 
\begin{proof}
Suppose, $ \dim( VG^1_{m,n})=2$ and $\{ u,v\}$ is a resolving set of $VG^1_{m,n}$. By Theorem~\ref{md2}, there is a unique shortest $u,v$-path  wherein $\deg_G(u)\le 3$, $\deg_G(v)\le 3$. The following cases need to be explored:\\ 
\textbf{Case 1}: $u$ and $v$ are on a same boundary. As there are four boundaries, We have four subcases as follows: \\
\textbf{Case 1.1}: $u=(2t+1,0)$ and $v=(2t+3,0) $, $ t\in [\![0,n-2]\!].$\\
Then $d(u,(2t+1,2)) = d(v,(2t+1,2))$ and $d(u,(2t+3,2)) = d(v,(2t+3,2))$, a contradiction.\\
\textbf{Case 1.2:}
$u=(0,2t+1)$ and $v=(0,2t+3)$, $ t\in [\![0,m-2]\!].$ \\
Then $d(u,(2,2t+1)) = d(v,(2,2t+1))$ and $d(u,(2,2t+3)) = d(v,(2,2t+3))$, a contradiction.\\
\textbf{Case 1.3:}
$u=(2t+1,2m)$ and $v=(2t+3,2m)$, $ t\in [\![0,n-2]\!].$ \\
We get $d(u,(2t+1,2m-2)) = d(v,(2t+1,2m-2))$ and $d(u,(2t+3,2m-2)) = d(v,(2t+3,2m-2))$, a contradiction.\\
\textbf{Case 1.4:}
$u=(2n,2t+1)$ and $v=(2n,2t+3)$, $ t\in [\![0,m-2]\!].$ \\
We have $d(u,(2n-2,2t+1)) = d(v,(2n-2,2t+1))$ and $d(u,(2n-2,2t+3)) = d(v,(2n-2,2t+3))$, a contradiction.\\
\textbf{Case 2:} $u$ and $v$ are the endpoints of the same obtuse line (refer to Figure \ref{F2}).\\
\textbf{Case 2.1:}
$u=(2t+1,0)$ and $v=(0,2t+1)$, $ t\in [\![0,m-1]\!].$ \\
Then  $d(u,(2t+3,0)) = d(v,(2t+3,0))$ and $d(u,(2t+3,2)) = d(v,(2t+3,2))$, a contradiction.\\
\textbf{Case 2.2:}
$u=(2t+1,0)$ and $v=(2t-2m+1,2m)$, $ t\in $[\![$m, n-2$]\!]$.$ \\
We get $d(u,(2t+3,0)) = d(v,(2t+3,0))$ and $d(u,(2t+3,2)) = d(v,(2t+3,2))$, a contradiction.\\ 
\textbf{Case 2.3:}
$u=(2n-1,0)$ and $v=(2n-2m-1,2m).$ \\
Then $d(u,(2n-2m-3,2m)) = d(v,(2n-2m-3,2m))$ and $d(u,(2n-2m-3,2m-2)) = d(v,(2n-2m-3,2m-2))$, a contradiction.\\
\textbf{Case 2.4:}
$u=(2n,2t+1)$ and $v=(2n-2m+2t+1,2m)$, $ t\in [\![0,m-1]\!].$ \\
We have $d(u,(2n-2m+2t-1,2m)) = d(v,(2n-2m+2t-1,2m))$ and $d(u,(2n-2m+2t-1,2m-2)) = d(v,(2n-2m+2t-1,2m-2))$, a contradiction.\\ 
\textbf{Case 3:} $u$ and $v$ are the endpoints of a same acute line.\\
\textbf{Case 3.1:}
$u=(0,2t+1)$ and $v=(2m-2t-1,2m)$, $ t\in [\![0,m-1]\!].$ \\
Then $d(u,(2m-2t+1,2m)) = d(v,(2m-2t+1,2m))$ and $d(u,(2m-2t+1,2m-2)) = d(v,(2m-2t+1,2m-2))$, a contradiction.\\
\textbf{Case 3.2:}
$u=(2t+1,0)$ and $v=(2m+2t+1,2m)$, $ t\in [\![0, n-m-2]\!].$ \\
We have $d(u,(2m+2t+3,2m)) = d(v,(2m+2t+3,2m))$ and $d(u,(2m+2t+3,2m-2)) = d(v,(2m+2t+3,2m-2))$, a contradiction.\\
\textbf{Case 3.3:}
$u=(2n-1,2m)$ and $v=(2n-2m-1,0).$ \\
We get $d(u,(2n-2m-3,0)) = d(v,(2n-2m-3,0))$ and $d(u,(2n-2m-3,2)) = d(v,(2n-2m-3,2))$, a contradiction.\\
\textbf{Case 3.4:}
$u=(2n-2t-1,0)$ and $v=(2n,2t+1)$, $ i\in [\![0,m-1]\!].$ \\
Then $d(u,(2n-2t-3,0)) = d(v,(2n-2t-3,0))$ and $d(u,(2n-2t-3,2)) = d(v,(2n-2t-3,2))$, a contradiction.\\
Therefore, we conclude that $ \dim( VG^1_{m,n}) >2$.
\end{proof}
\begin{lem} \label{lemma1}
If $n\geq 2$ and $m=n$, then $ \dim(VG^1_{m,n}) >2$.
 \end{lem} 
\begin{proof}
Suppose, $ \dim( VG^1_{m,n})=2$ and $\{ u,v\}$ is a resolving set for $VG^1_{m,n}$. By Theorem~\ref{md2} we consider the cases below. \\ 
\textbf{Case 1}: $u$ and $v$ are on a same boundary.  \\
\textbf{Case 1.1}: $u=(2t+1,0)$ and $v=(2t+3,0) $, $ t\in [\![0,n-2]\!].$\\
Then $d(u,(2t+1,2)) = d(v,(2t+1,2))$ and $d(u,(2t+3,2)) = d(v,(2t+3,2))$, a contradiction.\\
\textbf{Case 1.2:}
$u=(0,2t+1)$ and $v=(0,2t+3)$, $ t\in [\![0,m-2]\!].$ \\
Then $d(u,(2,2t+1)) = d(v,(2,2t+1))$ and $d(u,(2,2t+3)) = d(v,(2,2t+3))$, a contradiction.\\
\textbf{Case 1.3:}
$u=(2t+1,2m)$ and $v=(2t+3,2m)$, $ t\in [\![0,n-2]\!].$ \\
We get $d(u,(2t+1,2m-2)) = d(v,(2t+1,2m-2))$ and $d(u,(2t+3,2m-2)) = d(v,(2t+3,2m-2))$, a contradiction.\\
\textbf{Case 1.4:}
$u=(2n,2t+1)$ and $v=(2n,2t+3)$, $ t\in [\![0,m-2]\!].$ \\
We have $d(u,(2n-2,2t+1)) = d(v,(2n-2,2t+1))$ and $d(u,(2n-2,2t+3)) = d(v,(2n-2,2t+3))$, a contradiction.\\
\textbf{Case 2:} $u$ and $v$ are the endpoints of the same obtuse line.\\
\textbf{Case 2.1:}
$u=(2t+1,0)$ and $v=(0,2t+1)$, $ t\in [\![0,m-2]\!].$ \\
Then  $d(u,(2t+3,0)) = d(v,(2t+3,0))$ and $d(u,(2t+3,2)) = d(v,(2t+3,2))$, a contradiction.\\
\textbf{Case 2.2:}
$u=(2n+t-1,t)$ and $v=(t,2n+t-1)$, $ t\in $[\![$0,1$]\!]$ $. \\
We get $d(u,(2n+t-3,t)) = d(v,(2n+t-3,t))$ and $d(u,(2n+t-1,t+2)) = d(v,(2n+t-1,t+2))$, a contradiction.\\ 
\textbf{Case 2.3:}
$u=(2n,2t+1)$ and $v=(2t+1,2m)$, $ t\in [\![0,m-1]\!].$ \\
We have $d(u,(2t-1,2m-2)) = d(v,(2t-1,2m-2))$ and $d(u,(2t-1,2m)) = d(v,(2t-1,2m))$, a contradiction.\\ 
\textbf{Case 3:} $u$ and $v$ are the endpoints of a same acute line.\\
\textbf{Case 3.1:}
$u=(0,2t+1)$ and $v=(2m-2t-1,2m)$, $ t\in [\![1,m-1]\!].$ \\
Then $d(u,(2m-2t+1,2m)) = d(v,(2m-2t+1,2m))$ and $d(u,(2m-2t+1,2m-2)) = d(v,(2m-2t+1,2m-2))$, a contradiction.\\
\textbf{Case 3.2:}
$u=(t,1-t)$ and $v=(2m+t-1,2m-t)$, $ t\in [\![0, 1]\!].$ \\
We have $d(u,(t,3-t)) = d(v,(t,3-t))$ and $d(u,(t+2,1-t)) = d(v,(t+2,1-t))$, a contradiction.\\
\textbf{Case 3.3:}
$u=(2t+1,0)$ and $v=(2m,2m-2t-1)$, $ t\in [\![1,n-1]\!].$\\
We get $d(u,(2t-1,0)) = d(v,(2t-1,0))$ and $d(u,(2t-1,2)) = d(v,(2t-1,2))$, a contradiction.\\
Therefore, we conclude that $ \dim( VG^1_{m,n}) >2$.
\end{proof}
 The highlighted points in Figure \ref{F4} are the vertices that are not resolved by these two points $(1,0)$ and $(2n-1,0)$. Also, similar representations are given in similar colours. The following lemma demonstrates the usefulness of these points.

	\begin{figure}[H] 
	\centering
	\subfloat[]{\includegraphics[scale=.50]{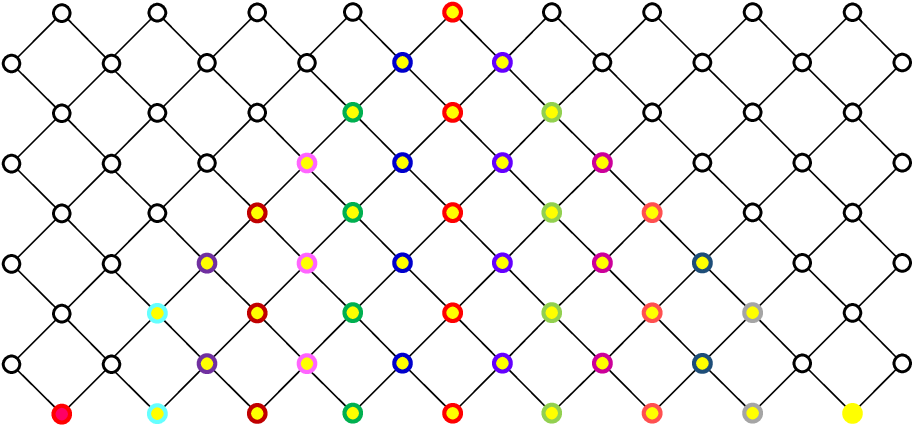}} 
	\quad 
	\subfloat[]{\includegraphics[scale=.50]{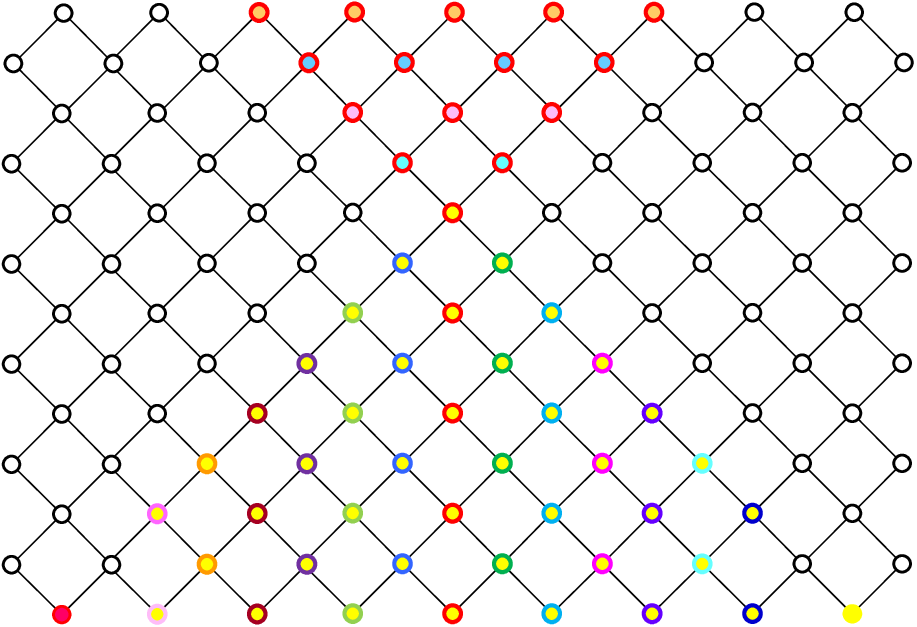}}
	\quad \quad \quad \quad \quad 
	\caption{(a) $VG^1_{4,9}$; (b) $VG^1_{6,9}$  } 
 \label{F4}
\end{figure}

\begin{lem} \label{lem2}
  For the graph $VG^1_{m,n}$ where $ 2 \leq n \leq 2m + 1$, the intersection $N_{r_1}((1,0)) \cap N_{r_2}((2n-1,0))$ is non-empty.
\end{lem} 
\begin{proof}

\textbf{Case 1:}
Let $r_1=2\nu-1$ and $r_2=2\gamma-1$, where $\nu, \gamma \in [\![1,n]\!].$ \\
\textbf{Case 1.1:} $\nu,  \gamma \in [\![1,m]\!]$.\\
$N_{r_1}((1,0)) = S_1 \cup S_2$, where 
$S_1= \{ (2t , r_1) : \enspace  t \in [\![0,\frac{r_1-1}{2}]\!] $\} and \\ 
$S_2 =\{ (r_1+
1 , r_1-2t) : \enspace  t \in [\![0,\frac{r_1-1}{2}]\!] $\}\\
$N_{r_2}((2n-1,0)) = S_1 \cup S_2$, where 
$S_1= \{ (2n-2t , r_2) : \enspace  t \in [\![0,\frac{r_2-1}{2}]\!]\}$ and \\
$S_2= \{ (2n-r_2-1 , r_2-2t) : \enspace  t \in [\![0,\frac{r_2-1}{2}]\!]\}$\\
\textbf{Case 1.2:} $\nu,  \gamma \in [\![m+1,n]\!].$\\
 $N_{r_1}((1,0)) = \{(r_1+1 , 2m-2t+1):   t \in [\![1,m]\!]\}$.\\
$N_{r_2}((2n-1,0)) = \{(2n-r_2-1 , 2m-2t+1):  t \in [\![1,m] \!]\}$\\ 
$N_{r_1}((1,0)) \cap N_{r_2}((2n-1,0))= A \cup B $  and  $ r_1=2\nu-1$, where $\nu \in [\![2,n-2]\!]$ 
$$A=\left\{
\begin{array}{ll}
\{(r_1+1 , r_1-2t): &  t \in [\![0,\frac{r_1-1}{2}]\!]\}   , \enspace \text{if}   \enspace  \nu \in [\![1,\lceil \frac{n-3}{2}\rceil]\!] ; \enspace r_2=2n-r_1-2\\
\{(r_1+1 , 2n-r_1-2t-2): &  t \in [\![0, \frac{2n-r_1-3}{2}]\!]\}, \enspace \text{if}      \enspace  \nu \in [\![\frac{n-1}{2},n-3]\!] ; \enspace r_2=2n-r_1-2\\ 
\end{array}
\right.$$
$$B=\left\{
\begin{array}{ll}
 \emptyset : &  n=2m+1\\
\{(2n-r_1+2t-1 , r_1): & \enspace t \in [\![0,r_1-n+1]\!]\}  , \enspace \text{if} \enspace n < 2m+1, \enspace   \nu \in [\![\frac{n-1}{2},m-1]\!] ; \enspace r_2=r_1\\ 
\end{array}
\right.$$
\textbf{Case 2:} $r_1=2\nu$  and $r_2=2\gamma$, where $\nu, \gamma \in [\![1,n-1]\!].$ \\
\textbf{Case 2.1:} $\nu , \gamma \in [\![1,m]\!]$.\\
 $N_{r_1}((1,0)) = S_1 \cup S_2$, where 
$S_1= \{(2t+1 , r_1) : \enspace  t \in [\![0,\frac{r_1}{2}]\!]\}$ and\\
$S_2= \{(r_1+1 , r_1-2t-2) : \enspace  t \in [\![0,\frac{r_1-2}{2}]\!]\}$\\
$N_{r_2}((2n-1,0)) = S_1 \cup S_2$, where 
$S_1= \{(2n-2t-1 , r_2) : \enspace  t \in [\![0,\frac{r_2-2}{2}]\!] \}$ and\\
$S_2= \{(2n-r_2-1 , r_2-2t) : \enspace  t \in [\![0,\frac{r_2}{2}]\!] \}$\\
\textbf{Case 2.2:} $\nu, \gamma \in [\![m+1,n-1]\!]$.\\
$N_{r_1}((1,0)) = \{(r_1+1 , 2m-2t):   t \in [\![0,m]\!]$\} \\
$N_{r_2}((2n-1,0)) =  \{(2n-r_2-1, 2m-2t):  t \in [\![0,m]\!]\}$\\
$N_{r_1}((1,0)) \cap N_{r_2}((2n-1,0))= A \cup B $  and  $r_1=2\nu$, where $ \nu \in [\![1,n-2]\!]$ 
$$A=\left\{
\begin{array}{ll}
\{(r_1+1 , r_1-2t): &  t \in [\![0,\frac{r_1}{2}]\!]\}  , \enspace \text{if} \enspace  \nu \in [\![1,\lfloor \frac{n-1}{2}\rfloor ]\!] ; \enspace r_2=2n-r_1-2\\
\{(r_1+1 , 2n-r_1-2t-2): &  t \in [\![0,\frac{2n-r_1-2}{2}]\!]\}, \enspace \text{if} \enspace  \nu \in [\![\lfloor \frac{n+1}{2}\rfloor,n-2]\!] ; \enspace r_2=2n-r_1-2\\ 
\end{array}
\right.$$
$$B=\left\{
\begin{array}{ll}
 \emptyset : &  n=2m+1\\
\{(2n-r_1+2t-1, r_1): &   , \enspace t \in [\![0,r_1-n+1]\!]\}, \enspace \text {if} \enspace  n < 2m+1, \enspace   \nu \in [\![\lfloor \frac{n+1}{2}\rfloor,m]\!]  ; \enspace r_2=r_1\\ 
\end{array}
\right.$$\\
In both cases, the intersection $N_{r_1}((1,0)) \cap N_{r_2}((2n-1,0))$ is non-empty, as shown by the inclusion of valid points in sets $A$ and $B$. Thus, we have proved that $N_{r_1}((1,0)) \cap N_{r_2}((2n-1,0))$ is indeed non-empty.
\end{proof}

 \begin{lem}\label{lem3} 
If $2 \leq n\leq 2m+1$ and $m \neq n$, then $ \dim( VG^1_{m,n})\leq 3$.
 \end{lem} 
\begin{proof}
We claim that $\{(1,0),(2n-1,0),(2m+1,2m)\}$ forms a  resolving set of the graph.\\
\textbf{Case 1:} $n=2m+1.$\\
\textbf{Case 1.1:} $r_3=2\nu-1$, where $\nu \in [\![1,\frac{n+1}{2}]\!].$\\
\textbf{Case 1.1.1:} $\nu =1. $\\
$N_{r_3}((2m+1,2m)) = S_1 \cup S_2$, where 
$S_1= \{(2m+r_3+1, 2m-2t-1):   t \in [\![0,\frac{r_3-1}{2}]\!]\}$ and 
$S_2= \{(2m-r_3+1, 2m-2t-1):  t \in [\![0,\frac{r_3-1}{2}]\!]\}$\\
\textbf{Case 1.1.2:} $\nu \in [\![2,\frac{n-1}{2}]\!].$\\
$N_{r_3}((2m+1,2m)) = S_1 \cup S_2\cup S_3$, where 
$S_1= \{(2m+2t-r_3+1, 2m-r_3) : \enspace  t \in [\![1,r_3-1]\!]$,\\ 
$S_2= \{(2m+r_3+1, 2m-2t-1):   t \in [\![0,\frac{r_3-1}{2}]\!]$ and \\
$S_3= \{(2m-r_3+1, 2m-2t-1): t \in [\![0,\frac{r_3-1}{2}]\!]\}$ \\
\textbf{Case 1.1.3:} $\nu = \frac{n+1}{2}.$\\
 $N_{r_3}((2m+1,2m)) = S_1 \cup S_2$, where 
$S_1= \{(2m+r_3+1, 2m-2t-1): t \in [\![0,m-1]\!]\} $ and\\ 
$S_2= \{(2m-r_3+1, 2m-2t-1):   t \in [\![0,m-1]\!]\}$\\ 
\textbf{Case 1.2:} $r_3=2\nu$, where $\nu \in [\![1,\frac{n-1}{2}]\!].$\\
$N_{r_3}((2m+1,2m)) = S_1 \cup S_2\cup S_3$, where 
$S_1= \{(2m+2t-r_3+1, 2m-r_3) : \enspace  t \in [\![1,r_3-1\!] \}$,\\
$S_2= \{(2m+r_3+1, 2m-2t) : \enspace  t \in [\![0,\frac{r_3}{2}]\!]\}$ and 
$S_3= \{(2m-r_3+1, 2m-2t) : \enspace  t \in [\![0,\frac{r_3}{2}]\!]\}$\\
\textbf{Case 2:}  $n < 2m+1 , n\neq m.$ \\
\textbf{Case 2.1:} $r_3=2\nu-1$, where $\nu \in [\![1,m+1]\!].$ \\
\textbf{Case 2.1.1:} $\nu=1$.\\
$N_r((2m+1,2m)) = S_1 \cup S_2$, where 
$S_1= \{(2m+r_3+1, 2m-2t-1) : \enspace  t \in [\![0,\frac{r_3-1}{2}]\!]\}$ and 
$S_2= \{(2m-r_3+1, 2m-2t-1):  t \in [\![0,\frac{r_3-1}{2}]\!]\}$\\
\textbf{Case 2.1.2:} $\nu \in [\![2,n-m]\!].$\\
$N_r((2m+1,2m)) = S_1 \cup S_2\cup S_3$, where 
$S_1= \{(2m+2t-r_3+1, 2m-r_3):  t \in [\![1,r_3-1]\!]\},$\\
$S_2= \{(2m+r_3+1, 2m-2t-1) : \enspace  t \in [\![0,\frac{r_3-1}{2}]\!]\}$ and 
$S_3= \{(2m-r_3+1, 2m-2t-1):  t \in [\![0,\frac{r_3-1}{2}]\!]\}$  \\
\textbf{Case 2.1.3:} $\nu \in [\![n-m+1,m]\!].$\\
$N_r((2m+1,2m)) = S_1 \cup S_2$, where 
$S_1= \{(2m+2t-r_3+1, 2m-r_3): t \in [\![1,\frac{r_3-1}{2}+n-m]\!] \}$ and\\
$S_2= \{(2m-r_3+1, 2m-2t-1): t \in [\![0,\frac{r_3-1}{2}]\!] \}$\\
\textbf{Case 2.1.4:} $\nu = m+1.$\\
$N_r((2m+1,2m)) = \{(2m-r_3+1, 2m-2t-1):  t \in [\![0,m-1]\!]\}$\\
\textbf{Case 2.2:} $r_3=2\nu$, where $\nu \in [\![1,m]\!].$ \\
\textbf{Case 2.2.1:} $\nu \in [\![1,n-m-1]\!].$\\
$N_r((2m+1,2m)) = S_1 \cup S_2\cup S_3$, where 
$S_1= \{(2m+2t-r_3+1, 2m-r):   t \in [\![1,r_3-1]\!]\},$\\
$S_2= \{(2m+r_3+1, 2m-2t) : \enspace  t \in [\![0,\frac{r_3}{2}]\!]\}$ and
$S_3= \{(2m-r_3+1, 2m-2t) : \enspace  t \in [\![0,\frac{r_3}{2}]\!]\}$\\
\textbf{Case 2.2.2:} $\nu \in [\![n-m,m]\!].$\\
$N_r((2m+1,2m)) = S_1 \cup S_2$, where 
$S_1= \{(2m+2t-r_3+1, 2m-r):   t \in [\![1,\frac{r_3}{2}+n-1-m]\!]\}$ and \\ 
$S_2= \{(2m-r_3+1, 2m-2t) : \enspace  t \in [\![0,\frac{r_3}{2}]\!]\}$\\
By  Lemma \ref{lem2},
in every instance mentioned above, $N_{r_1}(1,0)\cap N_{r_2}(2n+1,0)\cap N_{r_3}(2m+1,2m)$ is empty. Thus, $\{(1,0),(2n-1,0),(2m+1,2m)\}$ forms a resolving set. Hence,  $ \dim( VG^1_{m,n})\leq 3$.
\end{proof}
\begin{thm}\label{thm2}
 For $2\leq n \leq  2m+1$, $\dim (VG^1_{m,n}) =3$.
\end{thm}
\begin{proof}
 By  Lemmas \ref{lem1} and \ref{lem3}, we have $\dim (VG^1_{m,n}) =3$ if $n \neq m$.\\
 When $n = m$ and $n > 1$, by Lemma \ref{lemma1}, $\dim (VG^1_{m,n})\geq3$  and $\{(1,0),(2n-1,0),(2m,2m-1)\}$ forms a resolving set. The proof is analogous to those of Lemma \ref{lem3}, with some adjustments for the specific values of $n$ and $m$. Hence $\dim (VG^1_{m,n}) =3$. 
\end{proof}

 \begin{lem} \label{lem4}
If $n\leq 2m+1$, then $ \dim( VG^2_{m,n}) >2$.
 \end{lem} 
\begin{proof}
Suppose, for contradiction, that $\{ u,v\}$ is a resolving set for $VG^2_{m,n}$. By Theorem ~\ref{md2}, there is a unique shortest $u,v$-path and $\deg_G(u)\le 3$, $\deg_G(v)\le 3$. We consider the below cases:\\ 
\textbf{Case 1}: $u$ and $v$ are on the boundary.\\
\textbf{Case 1.1}:$u=(2t,0)$ and $v=(2t+2,0) $, $ t\in [\![0,n-2]\!]$.\\
Then $d(u,(2t,2)) = d(v,(2t,2))$ and $d(u,(2t+2,2)) = d(v,(2t+2,2))$, a contradiction.\\
\textbf{Case 1.2:}
$u=(0,2t)$ and $v=(0,2t+2)$, $ t\in [\![0,m-1]\!].$ \\
We get $d(u,(2,2t)) = d(v,(2,2t))$ and $d(u,(2,2t+2)) = d(v,(2,2t+2))$, a contradiction.\\
\textbf{Case 1.3:}
$u=(2t,2m)$ and $v=(2t+2,2m)$, $ t\in [\![0, n-2]\!].$ \\
We have $d(u,(2t,2m-2)) = d(v,(2t,2m-2))$ and $d(u,(2t+2,2m-2)) = d(v,(2t+2,2m-2))$, a contradiction.\\
\textbf{Case 1.4:}
$u=(2n-2,2t)$ and $v=(2n-2,2t+2)$, $ t\in [\![0,m-1]\!].$ \\
Now, $d(u,(2n-4,2t)) = d(v,(2n-4,2t))$ and $d(u,(2n-4,2t+2)) = d(v,(2n-4,2t+2))$, a contradiction.\\
\textbf{Case 2:} $u$ and $v$ are the endpoints of obtuse lines.\\
\textbf{Case 2.1:}
$u=(2t,0)$ and $v=(0,2t)$, $ t\in [\![1, m]\!].$ \\
Then $d(u,(2t+2,0)) = d(v,(2t+2,0))$ and $d(u,(2t+2,2)) = d(v,(2t+2,2))$, a contradiction.\\
\textbf{Case 2.2:}
$u=(2t,0)$ and $v=(2t-2m,2m)$, $ t\in [\![m+1, n-2]\!].$ \\
We have $d(u,(2t+2,0)) = d(v,(2t+2,0))$ and $d(u,(2t+2,2)) = d(v,(2t+2,2))$, a contradiction.\\ 
\textbf{Case 2.3:}
$u=(2n-2m+2t-2,2m)$ and $v=(2n-2,2t)$, $ t\in [\![0, m-1]\!].$ \\
Now, $d(u,(2n-2m+2t-4,2m)) = d(v,(2n-2m+2t-4,2m))$ and $d(u,(2n-2m+2t-4,2m-2)) = d(v,(2n-2m+2t-4,2m-2))$, a contradiction.\\ 
\textbf{Case 3:} $u$ and $v$ are the endpoints of acute lines.\\
\textbf{Case 3.1:}
$u=(0,2t)$ and $v=(2m-2t,2m)$, $ t\in [\![0,m-1]\!].$ \\
We get $d(u,(2m-2t+2,2m)) = d(v,(2m-2t+2,2m))$ and $d(u,(2m-2t+2,2m-2)) = d(v,(2m-2t+2,2m-2))$, a contradiction.\\
\textbf{Case 3.2:}
$u=(2t+2,0)$ and $v=(2m+2t+2,2m)$, $ t\in [\![0,n-m-2]\!].$ \\
Now, $d(u,(2t,0)) = d(v,(2t,0))$ and $d(u,(2t,2)) = d(v,(2t,2))$, a contradiction.\\
\textbf{Case 3.3:}
$u=(2n-2m+2t,0)$ and $v=(2n-2,2m-2t-2)$, $ t\in [\![0,m-2]\!].$ \\
Then $d(u,(2n-2m+2t-2,0)) = d(v,(2n-2m+2t-2,0))$ and $d(u,(2n-2m+2t-2,2)) = d(v,(2n-2m+2t-2,2))$, a contradiction.\\
We conclude that $ \dim( VG^2_{m,n}) >2$.
\end{proof}
\begin{lem} \label{lem5}
  For the graph $VG^2_{m,n}$ with $ n \leq 2m + 1$ , the intersection $N_{r_1}((0,0)) \cap N_{r_2}((2n-2,0))$ is non-empty.
\end{lem} 
\begin{proof}
We analyze the problem by considering different cases for $r_1$ and $r_2$, which represent specific rows in the grid of the graph $VG^2_{m,n}$.\\
\textbf{Case 1:} $r_1$ and $r_2$ are odd integers.\\
$r_1=2\nu-1$  and  $r_2=2\gamma-1$, where $\nu, \gamma \in [\![1,n-1]\!].$ Then the neighbourhoods $N_{r_1}((0,0))$ and $N_{r_2}((2n-2,0))$ are as follows.\\
\textbf{Case 1.1:} $\nu =1$ and $\gamma =1.$\\
$N_{r_1}((0,0)) = \{(2t+1 , r_1) : \enspace  t \in [\![0,\frac{r_1-1}{2}]\!]\}$\\ 
$N_{r_1}((2n-2,0)) = \{(2n-2t-3 , r_2) : \enspace  i \in [\![0,\frac{r_2-1}{2}]\!]\}$\\
\textbf{Case 1.2:} $\nu, \gamma \in [\![2,m]\!].$\\
$N_{r_1}((0,0)) = S_1 \cup S_2$, in which
$S_1= \{(2t+1 , r_1) : \enspace  t \in [\![0,\frac{r_1-1}{2}]\!]\}$ and \\
$S_2= \{(r_1 , r_1-2t): t \in [\![1,\frac{r_1-1}{2}]\!]\}$\\
$N_{r_1}((2n-2,0)) = S_1 \cup S_2$, in which
$S_1= \{(2n-2t-3 , r_2) : \enspace  t \in [\![0,\frac{r_2-1}{2}]\!]\}$ and \\
$S_2= \{(2n-r_2-2 , r_2-2t):  t \in [\![1,\frac{r_2-1}{2}]\!]\}$\\
\textbf{Case 1.3:} $\nu, \gamma \in [\![m+1,n-1]\!].$\\
$N_{r_1}((0,0)) = \{ (r_1 , 2m-2t+1): t \in [\![1,m]\!]\}$\\
$N_{r_1}((2n-2,0)) = \{ (2n-r_2-2 , 2m-2t+1): t \in [\![1,m] \!]\}$\\
$N_{r_1}((0,0)) \cap N_{r_2}((2n-2,0))= A \cup B $ and $ r_1=2\nu+1$ where $\nu \in [\![1,n-3]\!]$ \\
where 
$$A=\left\{
\begin{array}{ll}
\{(r_1, r_1-2t): &  t \in [\![0,\frac{r_1-1}{2}]\!]\}, \enspace \text{if}  \enspace  \nu\in [\![1,\lceil\frac{n-3}{2}\rceil]\!] ; \enspace r_2=2n-r_1-2\\
\{(r_1, 2n-r_1-2t-2): &  t \in [\![0,\frac{2n-r_1-3}{2}]\!]\}, \enspace \text{if}    \enspace  \nu \in [\![\lceil\frac{n-1}{2}\rceil,n-3]\!] ; \enspace r_2=2n-r_1-2\\ 
\end{array}
\right.$$
$$B=\left\{
\begin{array}{ll}
 \emptyset : &  n=1+2m\\
\{(2n-r_1+2t-2, r_1): &   \enspace t\in [\![0,r_1-n+1]\!]\}, \enspace \text{if} \enspace n < 2m+1 ,\enspace   \nu \in [\![\lceil\frac{n-1}{2}\rceil ,m-1]\!]  ; \enspace r_2=r_1\\ 
\end{array}
\right.$$
\textbf{Case 2:} $r_1$ and $r_2$ are even integers.\\
$r_1=2\nu$ and  $r_2=2\gamma$,  where $\nu, \gamma \in [\![1,n-1]\!].$\\
\textbf{Case 2.1:} $\nu, \gamma \in [\![1,m]\!]$.\\
$N_{r_1}((0,0)) = S_1 \cup S_2$, where 
$S_1= \{(2t, r_1) : \enspace  t \in [\![0,\frac{r_1}{2}]\!]\}$ and \\
$S_2= \{(r_1 , r_1-2t-2):  t \in [\![0,\frac{r_1-2}{2}]\!]\}$\\
$N_{r_2}((2n-2,0)) = S_1 \cup S_2$, where 
$S_1= \{(2n-2t-2 , r_2) : \enspace  t \in [\![0,\frac{r_2}{2}]\!]\}$ and \\
$S_2= \{(2n-r_2-2 , r_2-2t): t \in [\![1,\frac{r_2}{2}]\!]\}$\\
\textbf{Case 2.2:} $\nu, \gamma \in [\![m+1,n-1]\!] $.\\
$N_{r_1}((0,0)) = \{(r_1, 2m-2t):  t \in [\![0,m]\!]\}$\\
$N_{r_2}((2n-2,0)) = \{(2n-r_2-2 , 2m-2t):  t \in [\![0,m]\!]\}$\\
$N_{r_1}((0,0)) \cap N_{r_2}((2n-2,0))= A \cup B $ and $r_2=2\nu$, where $\nu \in [\![1,n-2]\!]$\\
if $n=2m+1 $  and $r_1=2\nu$, where $\nu \in [\![1,n-2]\!]$ then
$$A=\left\{
\begin{array}{ll}
\{(r_1, r_1-2t): &  t \in [\![0,\frac{r_1}{2}]\!]\}  , \enspace \text{if}  \enspace  \nu \in [\![1,\frac{n-1}{2}]\!] ; \enspace r_2=2n-r_1-2\\
\{(r_1, 2n-r_1-2t-2): &  t \in [\![0,\frac{2n-r_1-2}{2}]\!]\}, \enspace \text{if}   \enspace  \nu \in [\![\frac{n+1}{2},n-2]\!] ; \enspace r_2=2n-r_1-2\\ 
\end{array}
\right.$$
and $B$ is empty.\\
if $ n < 2m+1, n \neq m $ and $r_1=2\nu$, where $\nu \in [\![1,n-2]\!]$ \\
$$A=\left\{
\begin{array}{ll}
\{(r_1, r_1-2t): &  t \in [\![0,\frac{r_1}{2}]\!]\}, \enspace \text{if}    \enspace  \nu \in [\![1,\lfloor\frac{n-1}{2}\rfloor]\!] ; \enspace r_2=2n-r_1-2\\
\{(r_1, 2n-r_1-2t-2): &  t \in [\![0,\frac{2n-r_1-2}{2}]\!]\}, \enspace \text{if} \enspace  \nu \in [\![\lfloor\frac{n+1}{2}\rfloor,n-2]\!] ; \enspace r_2=2n-r_1-2\\ 
\end{array}
\right.$$
and 
$B= \{(2n-r_1+2t-2 , r_1): \enspace  t \in [\![0,r_1-n+1]\!]\} , \enspace \text{if}   \enspace  r_1 \in [\![\lfloor\frac{n+1}{2}\rfloor,m]\!] ; \enspace r_2=r_1$.\\
In both cases, the intersection $N_{r_1}((0,0)) \cap N_{r_2}((2n-2,0))$ is non-empty, as shown by the inclusion of valid points in sets $A$ and $B$. Thus, we have proved that $N_{r_1}((0,0)) \cap N_{r_2}((2n-2,0))$ is indeed non-empty.
\end{proof}

 \begin{lem}\label{lem6} 

If $n\leq 2m+1$, then $ \dim( VG^2_{m,n})\leq 3$.
 \end{lem} 
\begin{proof}
 We prove that $\{(0,0),(2n-2,0),(2m,2m)\}$ forms a resolving set of the graph.\\
 \textbf{Case 1:} $n=2m+1.$ \\
 \textbf{Case 1.1:} $r_3=2\nu-1$, where $\nu \in [\![1,\frac{n+1}{2}]\!].$\\
\textbf{Case 1.1.1:} $\nu =1. $\\
$N_{r_3}((2m,2m)) = S_1 \cup S_2$, where 
$S_1= \{(2m+r_3, 2m-2t-1):  \enspace  t \in [\![0,\frac{r_3-1}{2}]\!]\}$ and \\
$S_2= (2m-r_3, 2m-2t-1):   t \in [\![0,\frac{r_3-1}{2}]\!]\}$\\
\textbf{Case 1.1.2:} $\nu \in [\![2,\frac{n-1}{2}]\!].$\\
$N_{r_3}((2m,2m)) = S_1 \cup S_2\cup S_3$, where 
$S_1= \{(2m-r_3+2t, 2m-r_3) : \enspace  t \in [\![1,r_3-1]\!]$,\\ 
$S_2= \{(2m+r_3, 2m-2t-1):  \enspace  t \in [\![0,\frac{r_3-1}{2}]\!]$ and \\
$S_3= \{(2m-r_3, 2m-2t-1):   t \in [\![0,\frac{r_3-1}{2}]\!]\}$ \\
\textbf{Case 1.1.2:} $\nu = \frac{n+1}{2}.$\\
 $N_{r_3}((2m,2m)) = \{(2m-r_3+1, 2m-2t-1): t \in [\![0,m-1]\!]\} $\\ 
\textbf{Case 1.2:} $r_3=2\nu$, where $\nu \in [\![1,\frac{n-1}{2}]\!].$\\
$N_{r_3}((2m,2m)) = S_1 \cup S_2\cup S_3$, where 
$S_1= \{(2m-r_3+2t, 2m-r_3) : \enspace t \in [\![1,r_3-1]\!]\} $,\\
$S_2= \{(2m+r_3, 2m-2t) : \enspace  t \in [\![0,\frac{r_3}{2}]\!]\}$, and 
$S_3= \{(2m-r_3, 2m-2t) : \enspace  t \in [\![0,\frac{r_3}{2}]\!] \}$\\
\textbf{Case 2:} $n < 2m+1$, $n\neq m.$\\ 
\textbf{Case 2.1:} $r_3=2\nu-1$, where $\nu \in [\![1,m]\!].$\\
\textbf{Case 2.1.1:} $\nu=1.$\\
$N_{r_3}((2m,2m)) = S_1 \cup S_2$, where 
$S_1= \{(2m+r_3, 2m-2t-1) : \enspace  t \in [\![0,\frac{r_3-1}{2}]\!]\}$ and \\
$S_2= \{(2m-r_3, 2m-2t-1) : \enspace  t \in [\![0,\frac{r_3-1}{2}]\!] \}$\\
\textbf{Case 2.1.2:} $\nu \in [\![2,n-m-1]\!].$\\
$N_{r_3}((2m,2m)) = S_1 \cup S_2\cup S_3$, where 
$S_1= \{(2m-r_3+2t, 2m-r_3):   t \in [\![1,r_3-1]\!]\},$\\
$S_2= \{(2m+r_3, 2m-2t-1) : \enspace  t \in [\![0,\frac{r_3-1}{2}]\!]\}$ and 
$S_3= \{(2m-r_3, 2m-2t-1) : \enspace  t \in [\![0,\frac{r_3-1}{2}]\!] \}$\\
\textbf{Case 2.1.3:} $\nu \in [\![n-m,m]\!].$\\
$N_{r_3}((2m,2m)) = S_1 \cup S_2$, where 
$S_1= \{(2m-r_3+2t, 2m-r_3):   t \in [\![1,\frac{r_3-1}{2}+n-m-1]\!]\}$ and \\
$S_2= \{(2m-r_3, 2m-2t-1) : \enspace  t \in [\![0,\frac{r_3-1}{2}]\!]\!] \}$\\
\textbf{Case 2.2:} $r_3=2\nu$, where $\nu \in [\![1,m]\!].$\\
\textbf{Case 2.2.1:} $\nu \in [\![1,n-m-1]\!] $.\\
$N_{r_3}((2m,2m)) = S_1 \cup S_2\cup S_3$, where
$S_1= \{(2m-r_3+2t, 2m-r_3):   t \in [\![1,r_3-1]\!] \}$, \\
$S_2= \{(2m+r_3, 2m-2t) : \enspace  t \in [\![0,\frac{r_3}{2}]\!] \}$ and 
$S_3= \{(2m-r_3, 2m-2t) : \enspace t \in [\![0,\frac{r_3}{2}]\!] $\}\\
\textbf{Case 2.2.2:} $\nu \in [\![n-m,m]\!] $.\\
$N_{r_3}((2m,2m)) = S_1 \cup S_2$, where
$S_1= \{(2m-r_3+2t, 2m-r_3):  t \in [\![1,\frac{r_3}{2}+n-m-1]\!] \}$ and \\
$S_2= \{(2m-r_3, 2m-2t) : \enspace  t \in [\![0,\frac{r_3}{2}]\!] $\}\\
By Lemma \ref{lem5}, the intersection 
$N_{r_1}(0,0)\cap N_{r_2}(2n-2,0) \cap N_{r_3}(2m,2m)$ is empty.
Thus, $\{(1,0),(2n-1,0),(2m+1,2m)\}$ forms a resolving set for $VG^2_{m,n}$. Hence,  $ \dim( VG^2_{m,n})\leq 3$.
\end{proof}

\begin{thm}\label{mgthm4}
 For $n \leq  2m+1$, $\dim (VG^2_{m,n}) =3$.
\end{thm}
\begin{proof}
 By  Lemmas \ref{lem4} and \ref{lem6}, we get  $\dim (VG^2_{m,n})\leq 3$ and
$\dim (VG^2_{m,n}) \geq 3$. Therefore, $\dim (VG^2_{m,n}) =3$. 
\end{proof}
 \begin{conj} 
For $n > 2m+1$, $\dim (VG^1_{m,n})= \dim (VG^2_{m,n})= \lceil \frac{n-1}{m} \rceil +1.$
 \end{conj}
  Figure \ref{F5} and \ref{F6} provides evidence supporting this conjecture. Additionally, the correctness of the upper bound in Conjecture 1 for the case $n > 2m+1$ can generally be verified by considering the following sets. \\
  Define $R_1=\{(1+t(2m), m[1+(-1)^{t+1}])\quad; 0  \leq t \leq \lceil \frac{n-1}{m} \rceil -1 \} $. The set  $ R_1 \cup (x,y)$ forms a resolving set for $VG^1_{m,n}$ with size $\lceil \frac{n-1}{m} \rceil +1$.\\ 
  where, $$(x,y)=\left\{
\begin{array}{ll}
(2n, 2(n+m-\lceil \frac{n-1}{m} \rceil m)-1); & \text{if} \enspace \lceil \frac{n-1}{m} \rceil \enspace \text{is odd},  \\
(2n, 1+2(\lceil \frac{n-1}{m} \rceil m-n)); & \text{if} \enspace \lceil \frac{n-1}{m} \rceil \enspace \text{is even}.  \\ 
\end{array}
\right.$$
 
   \begin{figure}[H]
	\centering
\includegraphics[scale=0.5]{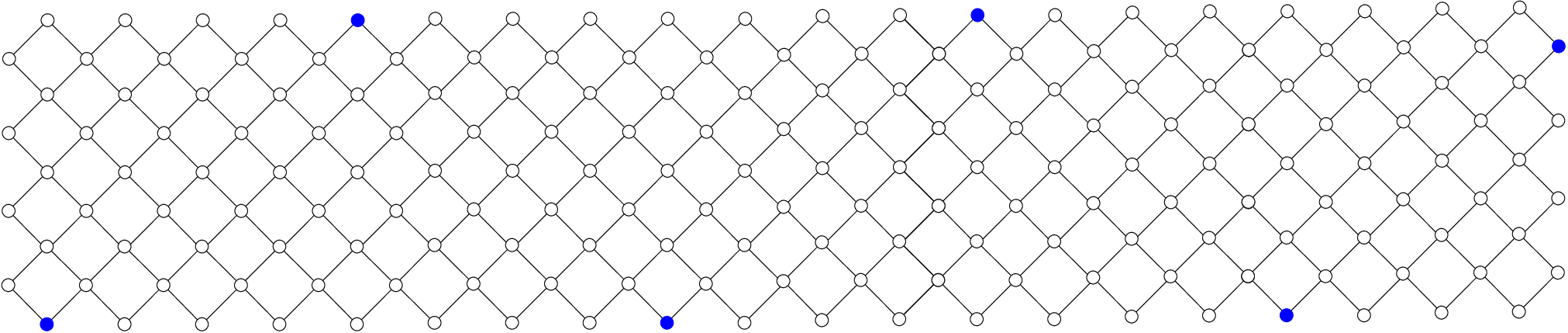}
	\caption{$VG^1_{4,20}$}
	\label{F5}
\end{figure}
Similarly, define $R_2=\{(2mt, m[1+(-1)^{t+1}])\quad; 0  \leq t \leq \lceil \frac{n-1}{m} \rceil -1 \} $. The set  $ R_2 \cup (x,y)$ is a resolving set for $VG^2_{m,n}$ with size $\lceil \frac{n-1}{m} \rceil +1$,\\ 
where, $$(x,y)=\left\{
\begin{array}{ll}
(2n-2, 2(n+m-\lceil \frac{n-1}{m} \rceil m-1); & \text{if} \enspace \lceil \frac{n-1}{m} \rceil \enspace \text{is odd},  \\
(2n-2, 2(\lceil \frac{n-1}{m} \rceil m-n+1)); & \text{if} \enspace \lceil \frac{n-1}{m} \rceil \enspace \text{is even}.  \\ 
\end{array}
\right.$$
\begin{figure}[H]
	\centering
\includegraphics[scale=0.5]{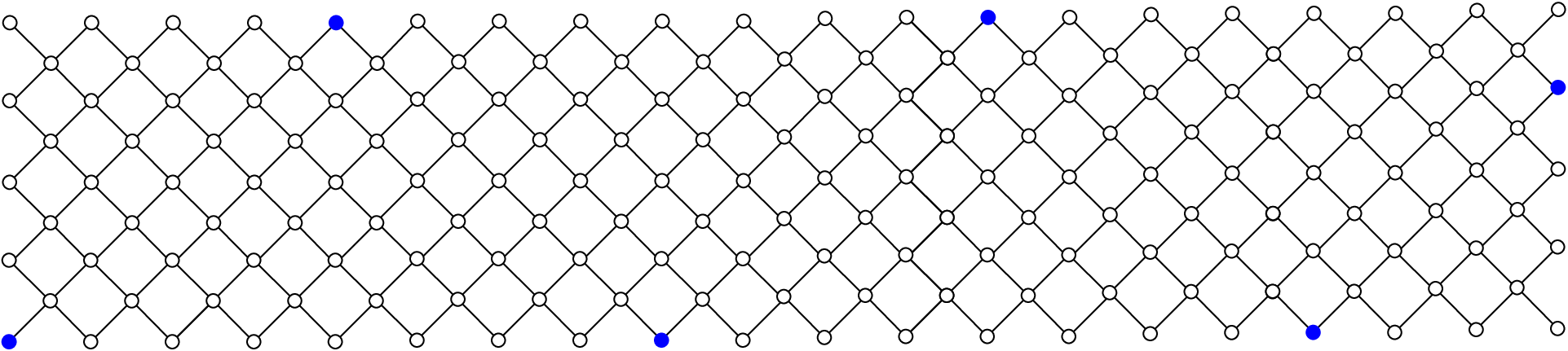}
	\caption{$VG^2_{4,20}$}
	\label{F6}
\end{figure}

\section{Conclusion}
The Villarceau grid is a geometric structure that emerges from the intersection of a cylinder and a plane, creating a unique pattern of circles on the surface of a torus. While it might seem like an abstract concept, its applications go beyond pure mathematics. This grid offers a hands-on way to explore and understand the complex shapes and properties of tori. In the world of computer graphics and visualization, the Villarceau grid can be used to create striking visual effects and simulate curved surfaces with precision. Its intricate design has the potential to inspire new ideas in architecture and design, where the principles behind the grid could lead to innovative structures and patterns. Furthermore, the grid is relevant in scientific fields where toroidal shapes play a crucial role, such as in the design of fusion reactors or the study of fluid dynamics. By providing a clearer understanding of these shapes, the Villarceau grid can contribute to advances in technology and engineering. In short, it is a tool that bridges the gap between abstract mathematics and real-world applications.\\
\textbf{Data availability:} No data associated with the manuscript.\\
\textbf{Declarations:}\\
 \textbf{Conflict of interest:} The authors declare that there is no conflict of interest regarding the publication of 
this paper.

\end{document}